\documentclass{article}
\usepackage[pdftex]{hyperref}
\usepackage{tikz-cd, url}

\hypersetup{
  colorlinks   = true, 
  urlcolor     = gray, 
  linkcolor    = red, 
  citecolor   =  blue
}
\usepackage{amsmath,amsfonts,amsthm,amssymb,graphicx,xcolor}
\usepackage[all,2cell,ps]{xy}

\newcommand{\subjclass}[1]{%
  \par\noindent\textit{2020 Mathematics Subject Classification.} #1.%
}

\newcommand{\keywords}[1]{%
  \par\noindent\textit{Key words and phrases.} #1.%
}
\newtheorem{theorem}{Theorem}[section]
\newtheorem{proposition}[theorem]{Proposition}
\newtheorem{lemma}[theorem]{Lemma}
\newtheorem{corollary}[theorem]{Corollary}
\theoremstyle{remark}

\newtheorem{definition}[theorem]{Definition}

\newcommand{\R}{\mathbb R}
\newcommand{\Sph}{\mathbb S}
\newcommand{\Bdef}{\mathcal B}
\newcommand{\Ndef}{\mathcal N_{\phi}}
\newcommand{\Qdef}{\mathcal Q}

\newcommand{\ord}{\operatorname{ord}}

\usepackage{mathtools}

\DeclarePairedDelimiter\abs{\lvert}{\rvert}
\DeclarePairedDelimiter\norm{\lVert}{\rVert}

\makeatletter
\let\oldabs\abs
\def\abs{\@ifstar{\oldabs}{\oldabs*}}
\let\oldnorm\norm
\def\norm{\@ifstar{\oldnorm}{\oldnorm*}}
\makeatother

\title{Almgren frequency and Agmon convexity on rotationally symmetric Cartan-Hadamard manifolds}
 \author{\small{Luca F. Di Cerbo} \\ \scriptsize{University of Florida} \\ \footnotesize{\textsf{ldicerbo@ufl.edu}}}
\date{}

\begin{document}

\maketitle

\begin{abstract}
We introduce a \emph{global} frequency function for harmonic functions on rotationally symmetric Cartan-Hadamard manifolds. In real hyperbolic space, we also identify a different frequency function associated with the normalized spherical $L^2$-mass of a hyperbolic harmonic function and prove its monotonicity. As a consequence, we obtain a logarithmic convexity theorem for the normalized spherical $L^2$-mass and a hyperbolic analog of the classical three-sphere theorem.

\end{abstract}

\vspace{8cm}




\tableofcontents\quad\\

\vspace{1.5cm}

\section{Introduction and Main Results}

In \cite{Alm79}, Almgren introduced, among many other things, a globally defined frequency function associated with a harmonic function on Euclidean space. This beautiful function has several remarkable monotonicity properties and is useful for studying the convexity of the $L^2$-mass of harmonic functions on spheres. This circle of ideas has been extended \emph{locally} to Riemannian manifolds and to general second-order differential operators; see \cite{GL86}.

The author has been interested for quite some time in the problem of finding \emph{global} generalizations of Almgren's frequency function to Riemannian manifolds. In particular, this study originates from an attempt to construct an Almgren frequency function on general Cartan-Hadamard manifolds. In Section \ref{Generalization1}, we obtain a global Almgren frequency function on rotationally symmetric Cartan-Hadamard manifolds. 

\begin{theorem}\label{frequency}
Let $n\geq 3$ and let $f$ be a nonzero harmonic function on $\R^n$ equipped with a rotationally symmetric metric $g_{\mathrm{rot}}=dr^2+\phi(r)^2g_{\Sph^{n-1}}$ of nonpositive sectional curvature (\emph{i.e.}, $\phi^{\prime\prime}\geq 0$, $\phi^\prime\geq 1$).  Then
\begin{equation}\label{eq:Nprime}
\Ndef(r)=\phi(r)\frac{\int_{B_r}\phi^\prime(s)|\nabla f|^2dv}{\int_{S_r}f^2d\sigma}
\end{equation}
is nondecreasing for $r>0$, and if $p$ denotes the pole of the rotationally symmetric metric, then $\lim_{r\to 0}\Ndef(r)=\ord_{p}(f)$.
\end{theorem}
This existence result follows from Theorem \ref{thm:Nmono} and Proposition \ref{order}. We refer to Section \ref{Generalization1} for further details and for a discussion comparing this frequency function with Almgren's original frequency function in Euclidean space. Remarkably, this frequency function is derived from the theory of the Price equality for harmonic forms, developed by the author in collaboration with M. Stern in \cite{DS22}. By considering the harmonic $1$-form naturally associated with a harmonic function, this theory suggests introducing the \emph{Price defect} of a harmonic function; see Definition \ref{defect}. The monotonicity properties of this quantity ultimately govern the entire Almgren monotonicity package.

Theorem \ref{frequency} provides a frequency function on the real hyperbolic space. Interestingly, this frequency function is \emph{not} well suited to studying the convexity properties of the normalized $L^2$-mass of harmonic functions on a geodesic sphere. Recall that in Euclidean space there is a rich theory concerning the convexity properties of harmonic functions going back to Agmon's influential monograph \cite{Agm66} and the three-sphere theorem due to Korevaar--Meyers; see \cite{KM94}. Motivated by this circle of ideas, we introduce a \emph{different} frequency function (\emph{cf}. Proposition \ref{NewN} and Proposition \ref{hyperbolicF}) on the real hyperbolic space that allows us to generalize these convexity results. For example, we have the following hyperbolic three-sphere theorem.
\begin{theorem}\label{hyperbolic-Agmon}
Let $f$ be harmonic on $\mathbb H^n$, $n\ge3$. Define $F(r)=\int_{S_r}f^2d\sigma$. Then
\begin{equation}\notag
\frac{d^2}{dx^2}
\log\!\left(
\frac{F(r)}{\sinh^{n-1}r}
\right)\ge0,
\qquad
x=\log\tanh r.
\end{equation}
Equivalently, if $0<r_1<r<r_2$ and $\theta\in(0,1)$ is determined by
\begin{equation}\notag
\log\tanh r
=(1-\theta)\log\tanh r_1
+\theta\log\tanh r_2,
\end{equation}
then
\begin{equation}\label{eq:3sphere-hyperbolic}
\frac{F(r)}{\sinh^{n-1}r}
\le
\left(
\frac{F(r_1)}{\sinh^{n-1}r_1}
\right)^{1-\theta}
\left(
\frac{F(r_2)}{\sinh^{n-1}r_2}
\right)^\theta.
\end{equation}
\end{theorem}
For a proof, a detailed discussion, and a review of the Euclidean theory, we refer to Section \ref{Agmon}, in particular Theorem \ref{hyperbolic-Agmon}.

\vspace{0.2 in}

\noindent\textbf{Acknowledgments}. 
The author is grateful to M. Stern for more than a decade of inspiring collaboration and for all that he has learned from him along the way. He is partially supported by the Simons Foundation grant MPS-TSM-00013711.

\section{Price Defect for Harmonic Functions}\label{Pricef}

In this section, we recall definitions and results from \cite[Section 2]{DS22}. We specialize these results to harmonic $1$-forms, and we introduce the \emph{Price defect} of a harmonic function.

The geometric setup relevant for our analysis is as follows. Let $(M^n,g)$ be a complete and connected Riemannian manifold with a pole $p\in M$. By definition, the exponential map $\exp_p : T_pM\to M$ is a diffeomorphism. Denote by $r$ the distance function from $p$:
\[
r(x)=d(p,x),
\]
and write
\[
B_r=B_r(p),\qquad S_r=\partial B_r(p).
\]
The function $r$ is smooth on $M\setminus\{p\}$. We denote by $h$ the second fundamental form of $S_r(p)$. Fix a local orthonormal frame $\{e_j\}^{n-1}_{j=1}$ tangent to $S_r(p)$, with dual coframe $\{\omega^j\}^{n-1}_{j=1}$. The radial vector field $\partial_r$ is orthogonal to $S_r(p)$; we denote its dual $1$-form by $ dr$. Given a $1$-form $\phi$, let $e(\phi)$ denote exterior multiplication on the left by $\phi$. Let $e^*(\phi)$ denote the adjoint operator. With this setup, recall that for $X$ and $Y$
tangent to $S_r$,
\[
h(X,Y):=g(\nabla_X\partial_r,Y).
\]
Thus, we write
\[
\nabla_{e_i}dr=h_{i1}\omega^1+\ldots+h_{in-1}\omega^{n-1},
\]
and
\begin{equation}\label{Q}
e(\omega^j)e^*(\nabla_{e_j}dr)=\sum_{j,k<n}h_{jk}e(\omega^j)e^*(\omega^k)=:Q.
\end{equation}
The operator $Q$ defined by Equation \eqref{Q} is the natural extension of the second fundamental
form to an endomorphism of forms of arbitrary degree. Here we need this operator only on $1$-forms. Finally, we set the
mean curvature of the geodesic sphere $S_r$ to be the trace $H=\sum_k h_{kk}$.

Next, we consider the harmonic $1$-form $\alpha=df$, where $f:M\to \R$ is harmonic function, \emph{i.e.,} $\Delta_g f=0$. Following \cite[Equations 11 \& 12]{DS22}, we define the following functions:
\begin{equation}\label{mu}
\mu_{df}(r):=\frac{\displaystyle\int_{S_r(p)}|e^*(dr)df|^2d\sigma}
{\displaystyle\int_{S_r(p)}|df|^2d\sigma},
\end{equation}

\begin{equation}\label{q}
q_{df}(r):=\frac{\displaystyle\int_{S_r(p)}\left\langle\left(\frac{H}{2}-Q\right)df,df\right\rangle d\sigma}
{\displaystyle\int_{S_r(p)}|df|^2d\sigma}.
\end{equation}
In this notation, \cite[Equation 13]{DS22} becomes
\begin{equation}\label{origindefect}
\int_{B_R(p)}q_{df}(r)|df|^2dv=\frac{1}{2}\int_{S_R(p)}(1-2\mu_{df}(R))|df|^2d\sigma.
\end{equation}
Equation \eqref{origindefect} is the Price-type identity, introduced by the author and Stern in \cite{DS22}, that underlies their formulation of the Donnelly-Xavier inequality \cite{DX84}; compare also \cite{Pri83}. This point of view suggests the following definition.
\begin{definition}\label{defect}
Let $f: M\to \R$ be a harmonic function on $(M,g)$. We define the real-valued function
\[
\Bdef(r):=\int_{S_r(p)}(1-2\mu_{df}(r))|df|^2d\sigma
\]
to be the \textbf{Price defect} of the harmonic function $f$.
\end{definition}

We conclude this section by studying the limiting behavior of $\mu_{df}(r)$ as $r\to 0$.  We denote by $k=\ord_p(f)$ the order of vanishing of $f$ at $p$. The next Lemma is a refinement of \cite[Lemma 18]{DS22}.

\begin{lemma}
Let $(M^n,g)$ be a Riemannian manifold and $p\in M$. Let $f: M\to \R$ be a harmonic function $k=\ord_p(f)$. We then have
\[
\lim_{r\to 0}\mu_{df}(r)=\frac{k}{2k+n-2}.
\]
\end{lemma}
\begin{proof}
Given $p\in M$, consider geodesic normal coordinates $(x_1, ..., x_n)$ around it. Let $r=|x|=d_{g}(p, x)$. In these coordinates, we have
\[
g_{ij}=\delta_{ij}+O(r^2), \quad \Gamma^k_{ij}(x)=O(r), \quad  \Delta_{g}=\Delta_{\mathrm{Euc}}+O(r)\partial+O(r^2)\partial^2,
\]
where 
\[
\Delta_{\mathrm{Euc}}=\sum^n_{i=1}\frac{\partial^2}{\partial x^2_i}.
\]
In these coordinates, we also have
\[
dv=(1+O(r^2))dv_{\mathrm{Euc}}, \quad d\sigma=(1+O(r^2))d\sigma_{\mathrm{Euc}}.
\]
Next, consider the Taylor expansion of $f$ centered at $p$ in geodesic normal coordinates. We have
\[
f(x_1, ..., x_{n}) = \sum_{i\geq k}P_i(x_1, ..., x_n), 
\]
where each $P_i$ is a homogeneous polynomials of degree $i$. Since $P_k$ is the first nonzero homogeneous term and $f$ is harmonic, $\Delta_{\mathrm{Euc}}P_k=0$. For harmonic homogeneous polynomials of degree $k$ in Euclidean space, we have the following identities:
\[
\int_{S_r}|\nabla P_k|^2d\sigma_{\mathrm{Euc}}=k(2k+n-2)\epsilon_k r^{2k+n-3}, \quad \int_{S_r}(\partial_r P_k)^2d\sigma_{\mathrm{Euc}}=k^2\epsilon_k r^{2k+n-3},
\]
where 
\[
\epsilon_k=\int_{S_1}P^2_k d\sigma_{\mathrm{Euc}}.
\]
We conclude that 
\[
\mu_{df}(r)=\frac{k}{2k+n-2}+O(r)
\]
and the proof is complete.
\end{proof}

\section{Almgren's Frequency Function: Background and Generalizations}\label{Generalization1}

In this section, we introduce several quantities inspired by Almgren's frequency function on $\R^n$. We derive the relation between those quantities and the functions $\Bdef(r)$, $q_{df}(r)$, and $\mu_{df}(r)$ (see Equations \eqref{mu}, \eqref{q}, and \eqref{origindefect}).

\subsection{Basic Identities}
Let $(M^n,g)$ be complete and simply connected with
\[
\sec_g\leq 0.
\]
By Cartan--Hadamard \cite[Theorem 6.2.2]{Petersen}, the exponential map at every point is a diffeomorphism, and every $p\in M$ is a pole.

Let $f: M\to\R$ be harmonic. On $S_r$ write
\[
f_r:=\partial_r f=\langle\nabla f,\partial_r\rangle,
\qquad
\nabla^T f:=\nabla f-f_r\partial_r.
\]
Thus
\[
|\nabla f|^2=f_r^2+|\nabla^T f|^2.
\]
We will use extensively the following quantities, defined for $r\geq 0$:
\begin{align}
F(r)&:=\int_{S_r}f^2d\sigma,\label{eq:F}\\
E(r)&:=\int_{B_r}|\nabla f|^2 dv,\label{eq:E}\\
A(r)&:=\int_{S_r}f_r^2 d\sigma,\label{eq:A}\\
T(r)&:=\int_{S_r}|\nabla^T f|^2 d\sigma,\label{eq:T}\\
D(r)&:=A(r)+T(r)=\int_{S_r}|\nabla f|^2 d\sigma.\label{eq:D}
\end{align}

The following proposition links the quantities above to the Price defect introduced in Section \ref{Pricef}.

\begin{proposition}\label{identity}
We have the following identities
\begin{equation}\label{eq:Bdef}
\mu_{df}(r)=\frac{A(r)}{D(r)}, \quad \Bdef(r):=T(r)-A(r)
=\int_{S_r}\bigl(|\nabla^T f|^2-f_r^2\bigr) d\sigma.
\end{equation}
\end{proposition}
\begin{proof}
First, notice that
\[
|\nabla f|^2=|df|^2,\quad  f^2_r=|e^*(dr)df|^2.
\]
Equation \eqref{mu} then gives the first identity. For the second, use the first identity and Definition \ref{defect} to obtain:
\[
\Bdef(r)=D(r)-2A(r)=T(r)-A(r).
\]
\end{proof}

Next, we derive an identity for the first derivative of the Price defect of a harmonic function.

\begin{proposition}\label{identity2}
Let $f: M\to\R$ be a harmonic function on $(M, g)$. We have
\begin{equation}\label{Bdefderivative}
\Bdef'(r)=2q_{df}(r)D(r).
\end{equation}
\end{proposition}
\begin{proof}
The co-area formula applied to Equation \eqref{origindefect} gives
\[
\Bdef'(r)=2q_{df}(r)\int_{S_r}|\nabla f|^2d\sigma=2q_{df}(r)D(r).
\]
\end{proof}

We conclude by stating two elementary relations between $E$, $D$, $F$, and $A$.

\begin{lemma}
We have the following
\begin{equation}\label{eq:Eprime}
E'(r)=D(r),
\end{equation}
and 
\begin{equation}\label{eq:CS}
E(r)^2\le F(r)A(r).
\end{equation}
\end{lemma}
\begin{proof}
The first identity follows from the co-area formula and the definition of $E(r)$. For the inequality, observe that Green's identity gives
\[
E(r)=\int_{B_r}|\nabla f|^2 dv
=\int_{S_r}f\partial_r f d\sigma-\int_{B_r}f\Delta f dv=\int_{S_r}ff_rd\sigma,
\]
since $f$ is harmonic. We conclude by applying the Cauchy--Schwarz inequality on $S_r$.
\end{proof}

\subsection{Price defect on rotationally symmetric manifolds}\label{rot}

We equip $\R^n$ with a complete rotationally symmetric metric
\[
g_{\mathrm{rot}}=dr^2+\phi(r)^2g_{\Sph^{n-1}}
\]
where
\[
\phi(0)=0,
\qquad
\phi'(0)=1,
\qquad
\phi(r)>0\quad(r>0).
\]
The radial and tangential sectional curvatures are
\begin{equation}\label{eq:modelcurv}
K_{\mathrm{rad}}(r)=-\frac{\phi''(r)}{\phi(r)},
\qquad
K_{\mathrm{tan}}(r)=\frac{1-(\phi'(r))^2}{\phi(r)^2};
\end{equation}
see \cite{Petersen}.  If the rotationally symmetric metric is Cartan--Hadamard (\emph{i.e.,} $sec_g\leq 0$), then
\[
\phi''\ge0,
\qquad
\phi'\ge1.
\]
Every geodesic sphere is totally umbilic, with
\begin{equation}\label{eq:lambda}
h|_{TS_r}=\lambda(r)g|_{TS_r},
\qquad
\lambda(r):=\frac{\phi'(r)}{\phi(r)}\ge0,
\end{equation}
and
\[
H(r)=(n-1)\lambda(r).
\]

\begin{proposition}\label{1std}
Let $n\ge3$ and let $f$ be harmonic on $(\R^n, g_{\mathrm{rot}})$.  Then
\begin{equation}\label{eq:Brot}
\Bdef'(r)
=\lambda(r)\bigl[(n-1)A(r)+(n-3)T(r)\bigr],
\end{equation}
and equivalently,
\begin{equation}\label{eq:Brot2}
\Bdef'(r)
=2(n-2)\lambda(r)A(r)+(n-3)\lambda(r)\Bdef(r).
\end{equation}
In particular, if $\lambda(r)\ge0$, then
\[
\Bdef'(r)\ge0.
\]
\end{proposition}

\begin{proof}
 We have
\begin{align*}
\frac{H(r)}{2}|d f|^2
-Q(df,df)&=\frac{H(r)}{2}|d f|^2
-Q(d^Tf,d^Tf)\\ \notag
&=\frac{(n-1)\lambda}{2}(f_r^2+|\nabla^Tf|^2)
-\lambda|\nabla^Tf|^2.
\end{align*}
Using Proposition \ref{identity2}, we compute
\[
\Bdef'(r)=2q_f(r)D(r)
=\lambda(r)\left[(n-1)A(r)+(n-3)T(r)\right].
\]
Since $T=A+\Bdef$, elementary rearrangement gives Equation \eqref{eq:Brot2}.  $A,T\ge0$, so $\Bdef'(r)$ is nonnegative whenever $\lambda\ge0$.
\end{proof}

More importantly, the normalized Price defect is also monotonic.

\begin{corollary}\label{cor:Qrot}
Let $n\ge3$ and let $f$ be harmonic on $(\R^n, g_{\mathrm{rot}})$. We have
\begin{equation}\label{eq:Qrot}
 \left(\phi(r)^{3-n}\Bdef(r)\right)'
 =2(n-2)\phi(r)^{2-n}\phi'(r)A(r).
\end{equation}
In particular, if $\lambda(r)\ge0$, then
$\phi(r)^{3-n}\Bdef(r)$ is nondecreasing.
\end{corollary}

\begin{proof}
Starting from Proposition~\ref{1std}, we have
\[
\Bdef'
=\lambda\bigl[(n-1)A+(n-3)T\bigr],
\qquad
\lambda=\frac{\phi'}\phi.
\]
Since $\Bdef=T-A$,
\begin{align*}
\left(\phi^{3-n}\Bdef\right)'
&=\phi^{3-n}\left[\Bdef'+(3-n)\lambda\Bdef\right]\\
&=\phi^{3-n}\lambda
\left[(n-1)A+(n-3)T-(n-3)(T-A)\right]\\
&=2(n-2)\phi^{3-n}\lambda A\\
&=2(n-2)\phi^{2-n}\phi' A.
\end{align*}
On a Cartan--Hadamard model, $\phi'\ge1$, so the right-hand side is nonnegative.
\end{proof}

\subsection{Almgren's frequency function in rotational symmetry}

By Corollary~\ref{cor:Qrot}, we have the natural monotone quantity
\[
\Qdef_{\mathrm{rot}}(r):=\phi(r)^{3-n}\Bdef(r).
\]
Let
\[
F(r)=\int_{S_r}f^2d\sigma
\]
as before, and define the normalized spherical mass of $f$
\begin{equation}\label{normmass}
\mathcal H(r):=\frac{F(r)}{\phi(r)^{n-1}}.
\end{equation}  
The first variation of $F(r)$ is easily computed
\begin{equation}\label{eq:Fprime-rot}
F'(r)
=2\int_{S_r}ff_rd\sigma
+(n-1)\lambda(r)F(r).
\end{equation}
Using Green's identity, we then have
\begin{equation}\label{Fprime}
F'(r)=2E(r)+(n-1)\lambda(r)F(r).
\end{equation}

We also need an integrated form of the defect equation.

\begin{lemma}\label{lem:intB}
On $(\R^n, g_{\mathrm{rot}})$, we have the following identities:
\begin{equation}\label{eq:phiBprime}
(\phi(r)\Bdef(r))'
=(n-2)\phi'(r)E'(r)
\end{equation}
\begin{equation}\label{eq:phiBint}
\phi(r)\Bdef(r)
=(n-2)\int_0^r\phi'(s)E'(s)ds
\end{equation}
\begin{equation}\label{eq:phiBparts}
\frac{\Bdef(r)}{n-2}
=\lambda(r)E(r)
-\frac1{\phi(r)}\int_0^r\phi''(s)E(s)ds.
\end{equation}
\end{lemma}

\begin{proof}
From Equation \eqref{eq:Brot2}, we obtain
\[
\Bdef'+\lambda\Bdef
=(n-2)\lambda(2A+\Bdef).
\]
But
\[
2A+\Bdef=2A+(T-A)=A+T=D=E'.
\]
Since $\phi\lambda=\phi'$, multiplication by $\phi$ gives Identity \eqref{eq:phiBprime}.  Because $E(0)=0$ and $\phi(0)=0$, integration from $0$ to $r$ proves Equation \eqref{eq:phiBint}.  Integrating the right-hand side by parts yields
\[
\phi\Bdef
=(n-2)\left(\phi'E-\int_0^r\phi''E ds\right),
\]
which is equivalent to Equation \eqref{eq:phiBparts}.
\end{proof}

Equation \eqref{eq:phiBint} suggests introducing a weighted Dirichlet energy. 

\begin{definition}[$\phi^\prime$-weighted Dirichlet energy]\label{def:EM}
For $n\ge3$, we define the $\phi^\prime$-weighted Dirichlet energy of a harmonic function $f$ by
\begin{equation}\label{eq:Ephiprime}
E_{\phi^\prime}(r)
:=\int_{B_r}\phi^\prime(s)|\nabla f|^2dv=\int^r_0\phi^\prime(s)E^\prime(s)ds.
\end{equation}
\end{definition}

Notice that in Euclidean space, $E_{\phi^\prime}$ reduces to the usual Dirichlet energy. We are now ready to introduce the global Almgren frequency function.

\begin{definition}[Price defect Almgren frequency]\label{def:N}
Let $n\ge3$ and let $f$ be harmonic on $(\R^n, g_{\mathrm{rot}})$. If $F(r)>0$, we define
\begin{equation}\label{eq:Ndef}
\Ndef(r)
:=\frac1{n-2}\frac{\phi(r)^{3-n}\Bdef(r)}{\mathcal H(r)}
=\frac{\phi(r)^2}{n-2}\frac{\Bdef(r)}{F(r)}.
\end{equation}
\end{definition}

The relation between this frequency function and Almgren's original frequency function on Euclidean space is better understood in terms of the $\phi^\prime$-weighted Dirichlet energy. More precisely, Equation \eqref{eq:phiBint} gives that
\[
\phi(r)\Bdef(r)=(n-2)E_{\phi^{\prime}}(r).
\]
Thus, we obtain the tantalizing identity
\[
\Ndef(r)=\phi(r)\frac{E_{\phi^{\prime}}(r)}{F(r)}
\]
that clearly reduces to the Euclidean Almgren's frequency function on Euclidean space; see Section \ref{Eucl} for a more detailed discussion. We now show that $\Ndef(r)$ obeys a monotonicity formula on rotationally symmetric Cartan-Hadamard manifolds.

\begin{theorem}\label{thm:Nmono}
Let $n\geq 3$ and let $f$ be a nonzero harmonic function on $(\R^n, g_{\mathrm{rot}})$.  Then
\begin{equation}\label{eq:Nprime}
\Ndef'(r)
=
\frac{2\phi(r)\phi'(r)}{F(r)^2}
\bigl(A(r)F(r)-E(r)^2\bigr)
+
\frac{2\phi(r)E(r)}{F(r)^2}
\int_0^r\phi''(s)E(s)ds.
\end{equation}
Hence, if
\[
\phi'\ge0,
\qquad
\phi''\ge0,
\]
then
\[
\Ndef'(r)\ge0.
\]
In particular, $\Ndef$ is nondecreasing on every Cartan--Hadamard $(\R^n, g_{\mathrm{rot}})$.
\end{theorem}

\begin{proof}
Differentiate \eqref{eq:Ndef}.  Since $\lambda=\phi'/\phi$,
\begin{align*}
\Ndef'
&=\frac{\phi^2}{(n-2)F}
\left[
\Bdef'+\left(2\lambda-\frac{F'}F\right)\Bdef
\right].
\end{align*}
Use Equations \eqref{Fprime} and \eqref{eq:Brot2} to write:
\[
\frac{F'}F=(n-1)\lambda+2\frac EF,
\]
\[
\Bdef'=2(n-2)\lambda A+(n-3)\lambda\Bdef.
\]
Substituting these identities into the formula for $\Ndef'$ yields
\begin{align*}
\Ndef'
&=\frac{\phi^2}{(n-2)F}
\left[
2(n-2)\lambda A-2\frac EF\Bdef
\right]\\
&=\frac{2\phi^2}{F}
\left[
\lambda A-\frac{E\Bdef}{(n-2)F}
\right].
\end{align*}
Now substituting Equation \eqref{eq:phiBparts}
\[
\frac{\Bdef}{n-2}
=\lambda E-\frac1\phi\int_0^r\phi''(s)E(s)ds
\]
yields
\begin{align*}
\Ndef'
&=\frac{2\phi^2}{F}
\left[
\lambda A-\lambda\frac{E^2}{F}
+\frac{E}{\phi F}\int_0^r\phi''(s)E(s)ds
\right]\\
&=
\frac{2\phi\phi'}{F^2}(AF-E^2)
+
\frac{2\phi E}{F^2}\int_0^r\phi''(s)E(s)ds,
\end{align*}
which is the identity claimed in Equation \eqref{eq:Nprime}.  The first term is nonnegative by Equation \eqref{eq:CS}, and the second is nonnegative when $\phi''\ge0$ because $E\ge0$.
\end{proof}

Next, we study the limit of $\Ndef(r)$ as $r\to 0$. Given $p\in M$ and a harmonic function $f$ on $M$, we denote by $\ord_p(f)$ the order of vanishing of $f$ at $p$.

\begin{proposition}[Order of vanishing]\label{order}
Let $n\geq 3$. If $f$ is nonzero harmonic and $k=\ord_{p}(f)$, then
\begin{equation}\label{eq:Nlimit}
\lim_{r\to 0}\Ndef(r)=k.
\end{equation}
\end{proposition}

\begin{proof}
If $k=0$, then $f(p)\ne0$.  One has $F(r)\sim c_0r^{n-1}$ while $E_{\phi^\prime}(r)=O(r^{n-1})$, so the limit is zero.  Assume $k\ge1$.  Smoothness of the rotationally symmetric metric at the pole gives
\[
\phi(r)=r+O(r^3).
\]
For harmonic homogeneous polynomials of degree $k$ in Euclidean space, we have the following identities:
\[
\int_{B_r}|\nabla P_k|^2dv_{\mathrm{Euc}}=k\epsilon_k r^{2k+n-2}, \quad \int_{S_r}P_k^2d\sigma_{\mathrm{Euc}}=\epsilon_k r^{2k+n-1},
\]
where 
\[
\epsilon_k=\int_{S_1}P^2_k d\sigma_{\mathrm{Euc}}.
\]
Thus, we have
\[
F(r)=\epsilon_k r^{2k+n-1}(1+O(r)),
\]
\[
E_{\phi^\prime}(r)=k\epsilon_k r^{2k+n-2}(1+O(r)).
\]
Therefore
\[
\Ndef(r)=k+O(r),
\]
which proves Equation \eqref{eq:Nlimit}.
\end{proof}

\subsection{Euclidean interpretation of the normalized defect and frequency}\label{Eucl}

In this section, we retrace the construction of the frequency function on Euclidean space $\R^n$ with the standard metric, \emph{i.e.}, when $\phi(r)=r$. Along the way, we compare the Price defect approach with the usual derivation of the frequency function presented in the standard analysis textbooks (see, for example, \cite[Chapter 6]{CM11}).

In Euclidean space, the Almgren frequency function is usually introduced as follows. Given a Euclidean harmonic function $f:\R^n\to \R$, one considers the scale-invariant Dirichlet energy and the normalized spherical $L^2$-mass of $f$:
\[
\frac{1}{r^{n-2}}\int_{B_r}|\nabla f|^2dv_{\mathrm{Euc}},\quad  \frac{1}{r^{n-1}}\int_{S_r}f^2d\sigma_{\mathrm{Euc}}.
\]
Almgren's frequency function is then introduced as the ratio of the scale-invariant Dirichlet energy to the normalized spherical mass of $f$. The monotonicity of the scale-invariant Dirichlet energy is derived as a consequence of the Rellich identity (see, for example, \cite[Equation 6.10]{CM11}). More precisely, Rellich's formula \cite{Rel40} provides a crucial identity for the derivative of the scale-invariant Dirichlet energy that is key for the proof of the monotonicity of the frequency function.

The Price defect approach starts with Equation \eqref{origindefect} together with the definition of the Price defect of a harmonic function $\Bdef(r)$. The normalized Price defect satisfies
\begin{equation}\label{eq:Qrot-Euclidean}
\Qdef_{\mathrm{rot}}(r)
:=r^{3-n}\Bdef(r)
=(n-2)r^{2-n}E(r).
\end{equation}
Thus, up to the constant factor $n-2$, $\Qdef_{\mathrm{rot}}$ is exactly the classical scale-invariant Dirichlet energy $r^{2-n}E(r)$. The monotonicity of $\Qdef_{\mathrm{rot}}(r)$ is a direct consequence of Equation \eqref{origindefect} (see Proposition \ref{identity2}). On $\R^n$, we have

\[
\mathcal H_{\R^n}(r):=\frac{F(r)}{r^{n-1}},
\]
so that
\begin{equation}\label{eq:N-Euclidean-quotient}
\Ndef(r)
=\frac1{n-2}\frac{\Qdef_{\mathrm{rot}}(r)}{\mathcal H_{\R^n}(r)}
=r\frac{E(r)}{F(r)},
\end{equation}
which is exactly the classical Almgren frequency. The Price defect proof of the monotonicity of this quantity follows by direct computation together with an integrated form of the defect equation; see Lemma \ref{lem:intB}. In particular, the whole monotonicity package follows from Equation \eqref{origindefect} and its differentiated and integrated versions.

One may wonder if this circle of ideas recovers the Rellich identity as well. This is the case. Indeed, one can observe that the integrated version of the Price defect in Euclidean space, see Equation \eqref{eq:phiBint} in Lemma \ref{lem:intB}, is nothing other than Rellich's formula:
\begin{align*}
r\Bdef(r)&=r\int_{S_r}|(\nabla^Tf|^2-f^2_r)d\sigma_{\mathrm{Euc}} \\ 
&=r\int_{S_r}(|\nabla f|^2-2f^2_r)d\sigma_{\mathrm{Euc}}=(n-2)\int_{B_r}|\nabla f|^2dv_{\mathrm{Euc}}.
\end{align*}
This is exactly \cite[Equation 6.10]{CM11}. In conclusion, Equation \eqref{origindefect} contains the whole monotonicity package, and it provides a \emph{global} frequency function not only on Euclidean space, but more generally on any rotationally symmetric Cartan-Hadamard manifold. Along the way, it also recovers the appropriate Rellich identity on these spaces.

\section{Agmon and Korevaar--Meyers convexity in rotational symmetry}\label{Agmon}

In this section, we seek an appropriate generalization of the classical logarithmic convexity of the $L^2$-mass of a Euclidean harmonic function on rotationally symmetric Cartan-Hadamard manifolds. This circle of ideas suggests a different frequency function than the one given in Theorem \ref{thm:Nmono}. While we are not able to prove that this new frequency is always monotonic, the theory of the Price inequality and the estimate on the Price defect we developed in Section \ref{Generalization1} are enough to prove monotonicity for certain choices of $\phi$. This result covers, in particular, the case $\phi(r)=\sinh(r)$, yielding a different frequency function on real hyperbolic space. This frequency function allows us to study the convexity properties of the $L^2$-mass of hyperbolic harmonic functions and to derive a hyperbolic three-sphere theorem.

\subsection{Agmon and Korevaar-Meyers convexity in Euclidean space}

We recall the classical logarithmic convexity theorem for the normalized spherical $L^2$-mass of a Euclidean harmonic function. This result is often referred to as the \emph{three-spheres inequality}; see, for example, the Wikipedia entry \href{https://en.wikipedia.org/wiki/Three_spheres_inequality}{``Three spheres inequality''}.

The classical reference for convexity properties of harmonic functions on Euclidean space seems to date back to 
Agmon's influential monograph \cite{Agm66}. The logarithmic convexity for the normalized $L^2$-mass for harmonic functions, which we recall below, is due to Korevaar--Meyers; see
\cite{KM94}, and also Brummelhuis \cite{Bru93}.  Let $f$ be harmonic on a Euclidean ball and write
\[
F(r)=\int_{S_r}f^2d\sigma_{\mathrm{Euc}},
\qquad
\mathcal H_{\R^n}(r)=\frac{F(r)}{r^{n-1}}.
\]
This is the normalized spherical $L^2$-mass of $f$ on $S_r$.

\begin{theorem}[Agmon and Korevaar-Meyers convexity]\label{thm:Agmon-Euclidean}
Let $f$ be harmonic on $\R^n$. Then
\begin{equation}\label{eq:Agmon-Euclidean-convexity}
\frac{d^2}{d(\log r)^2}\log \mathcal H_{\R^n}(r)\ge0,
\qquad 0<r<\infty.
\end{equation}
Equivalently, if $0<r_1<r<r_2<\infty$ and $\theta\in(0,1)$ is determined by
\[
\log r=(1-\theta)\log r_1+\theta\log r_2,
\]
then
\begin{equation}\label{eq:Agmon-Euclidean-three-sphere}
\mathcal H_{\R^n}(r)
\le
\mathcal H_{\R^n}(r_1)^{1-\theta}
\mathcal H_{\R^n}(r_2)^\theta.
\end{equation}
\end{theorem}

Equation \eqref{eq:Agmon-Euclidean-three-sphere} is sometimes referred to as the three-sphere theorem. The point of theoretical interest is that Theorem \ref{thm:Agmon-Euclidean} is an immediate consequence of
Almgren monotonicity. Indeed, Equation \eqref{Fprime} with $\phi(r)=r$ gives
\[
\frac{d}{dr}\log\mathcal H_{\R^n}(r)=2\frac{E(r)}{F(r)},
\]
and therefore
\begin{equation}\label{eq:Agmon-slope-Euclidean}
\frac12\frac{d}{d\log r}\log\mathcal H_{\R^n}(r)
=r\frac{E(r)}{F(r)}.
\end{equation}
Thus the logarithmic derivative of the normalized mass is exactly Almgren's
frequency, and Theorem \ref{thm:Agmon-Euclidean} follows from the monotonicity of Almgren's frequency function.

\subsection{A convexity-inspired frequency function}

Consider  $\R^n$ with a rotationally symmetric metric
\[
g_{\mathrm{rot}}=dr^2+\phi(r)^2g_{\Sph^{n-1}}
\]
where
\[
\phi(0)=0,
\qquad
\phi'(0)=1,
\qquad
\phi(r)>0\quad(r>0).
\]
If the rotationally symmetric metric is Cartan--Hadamard (\emph{i.e.,} $sec_g\leq 0$), then
\[
\phi''\ge0,
\qquad
\phi'\ge1.
\]
For a harmonic function $f: (\R^n, g_{\mathrm{rot}})\to \R$, we define
\[
F(r)=\int_{S_r}f^2d\sigma,
\qquad
E(r)=\int_{B_r}|\nabla f|^2dv,
\]
and set
\begin{equation}\label{eq:H-hyperbolic}
\mathcal H_{\phi}(r):=\frac{F(r)}{\phi(r)^{n-1}}.
\end{equation}
Up to the fixed constant, this is the mean-square value of $f$ on $S_r$.

For a general model function $\phi$, Equation
\eqref{Fprime} and Definition \eqref{normmass} give
\begin{equation}\label{log-derivative}
\frac{d}{dr}\log\mathcal H_{\phi}(r)=2\frac{E(r)}{F(r)}.
\end{equation}
Consequently, it is tantalizing to introduce a new radial coordinate $\tau_\phi$ by setting
\begin{equation}\label{eq:xphi-Agmon}
\frac{d\tau_\phi}{dr}=\frac1{\phi(r)\phi^\prime(r)},
\end{equation}
so that
\[
\frac{d}{d\tau_\phi}=\phi\phi^\prime\frac{d}{dr}.
\]
We then have
\begin{equation}\notag
\frac12\frac{d}{d\tau_\phi}\log\mathcal H_{\phi}(r)
=\phi(r)\phi'(r)\frac{E(r)}{F(r)}.
\end{equation}
This motivates the introduction of a different frequency function
\begin{equation}\label{newF}
\widetilde{\Ndef}(r)
:=\phi(r)\phi'(r)\frac{E(r)}{F(r)}.
\end{equation}
$\widetilde{\Ndef}$ arises as the logarithmic derivative of the normalized $L^2$-mass of $f$ over a sphere, and as such we have $\widetilde{\Ndef}^\prime(r)\geq 0$ if and only if $\log\mathcal H_{\phi}$ is convex as a function of $\tau_\phi$. Thus, if one could establish the monotonicity of the new frequency function defined in Equation \eqref{newF}, then the convexity arguments of Agmon and Korevaar--Meyers could be transplanted verbatim to rotationally symmetric Cartan--Hadamard manifolds. Unfortunately, such monotonicity does not hold for arbitrary choices of $\phi(r)$. Nevertheless, under suitable geometric assumptions, we obtain the following monotonicity result, whose proof relies on the Price defect.

\begin{proposition}\label{NewN}
Let $f$ be a nonzero harmonic function on a Cartan-Hadamard $(\R^n, g_{\mathrm{rot}})$.  
If
\[
\Big(\frac{\phi'}{\phi}\Big)^\prime\leq0,
\qquad
\phi^{\prime\prime}(r)\phi^{n-2}(r)\geq (n-2)\phi^\prime(r)\int^r_0\phi^{\prime\prime}(s)\phi^{n-3}(s)ds
\]
for all $r>0$, then
\[
\widetilde{\Ndef}^\prime(r)\ge0.
\]
\end{proposition}
\begin{proof}
If $f$ is constant, then $E=\widetilde{\Ndef}=0$, so there is nothing to prove. Hence assume that $f$ is nonconstant; then $E(r)>0$ for every $r>0$. Using Equation \eqref{Fprime}, we obtain
\[
\widetilde{\Ndef}^\prime(r)=\phi(r)\phi^\prime(r)\Big(\frac{E^\prime F-2E^2}{F^2}\Big)+\Big((\phi\phi^\prime)^\prime - \phi\phi^\prime(n-1)\frac{\phi^\prime}{\phi}\Big)\frac{E}{F}.
\]
We have
\[
\frac{E^\prime}{F}=\frac{2A}{F}+\frac{\Bdef}{F},
\]
so that
\[
\widetilde{\Ndef}^\prime(r)=2\phi(r)\phi^\prime(r)\Big(\frac{A F-E^2}{F^2}\Big)+\phi\phi^\prime\frac{\Bdef}{F}+\Big((\phi\phi^\prime)^\prime - (n-1)(\phi^\prime)^2\Big)\frac{E}{F}.
\]
Using Equation \eqref{eq:phiBparts} and the identity
\[
(\phi\phi^\prime)^\prime=(\phi^\prime)^2+\phi\phi^{\prime\prime}
\]
we obtain
\begin{equation}\label{intermediate}
\widetilde{\Ndef}^\prime(r)=2\phi(r)\phi^\prime(r)\Big(\frac{A F-E^2}{F^2}\Big)-(n-2)\frac{\phi^\prime}{F}\int^r_0\phi^{\prime\prime}Eds+\phi\phi^{\prime\prime}\frac{E}{F}.
\end{equation}
Equation \eqref{eq:Brot2} implies that
\[
\Bdef^\prime(r)-(n-3)\lambda(r)E^\prime(r)=2\lambda(r)A(r)\geq 0, \lambda(r)=\frac{\phi^\prime(r)}{\phi(r)},
\]
so that, using $\Bdef(r)\to 0$ and $\lambda(r)E(r)\to 0$ as $r\to 0$,
\[
\Bdef(r)\geq (n-3)\int^r_0\lambda(s)E^\prime(s)ds.
\]
Under the assumprion $\lambda^\prime\leq 0$, integration by parts gives
\[
\Bdef(r)\geq (n-3)\lambda(r)E(r).
\]
Next, since
\[
E^\prime=2A+\Bdef\geq \Bdef\geq (n-3)\lambda E
\]
we have
\[
\frac{E^\prime}{E}\geq (n-3)\frac{\phi^\prime}{\phi}
\]
and therefore $E(r)/\phi^{n-3}(r)$ is nondecreasing. Using this fact, and since on a rotationally symmetric Cartan-Hadamard manifold we have $\phi^{\prime\prime}\geq 0$, we have the estimate
\[
\int^r_0\phi^{\prime\prime}(s)E(s)ds\leq \frac{E(r)}{\phi(r)^{n-3}}\int^r_0\phi^{\prime\prime}(s)\phi(s)^{n-3}ds.
\]
Using this estimate in Equation \eqref{intermediate}, we obtain
\[
\widetilde{\Ndef}^\prime(r)\geq 2\phi(r)\phi^\prime(r)\Big(\frac{A F-E^2}{F^2}\Big)+\frac{E}{F}\Big(\phi\phi^{\prime\prime} -\frac{(n-2)\phi^\prime}{\phi^{n-3}}\int^r_0\phi^{\prime\prime}\phi^{n-3}ds\Big).
\]
By Equation \eqref{eq:CS} the first term is always nonnegative, and the second term is nonnegative if 
\[
\phi^{\prime\prime}(r)\phi^{n-2}(r)\geq (n-2)\phi^\prime(r)\int^r_0\phi^{\prime\prime}(s)\phi^{n-3}(s)ds.
\]
The proof is complete.
\end{proof}

Proposition \ref{newF} identifies a new Almgren frequency function on the real hyperbolic space. Indeed, $\phi(r)=\sinh(r)$ satisfies all of the requirements. First, observe that 
\[
\Big(\frac{\cosh(r)}{\sinh(r)}\Big)^\prime=-\frac{1}{\sinh^2(r)}<0.
\]
Second, if we define
\[
D_n(r)=\sinh^{n-1}(r)-(n-2)\cosh(r)\int^r_0\sinh^{n-2}(s)ds
\]
Differentiating and then integrating by parts, we obtain
\begin{align}\notag
D^\prime_n(r)&=\sinh^{n-2}(r)\cosh(r)-(n-2)\sinh(r)\int^r_0\sinh^{n-2}(s)ds\\ \notag
&=(n-3)\sinh(r)\int^r_0\sinh^{n-4}(s)ds>0
\end{align}
for $n\geq 4$. Since $D_{n}(0)=0$, it follows that  $D_{n}(r)\geq 0$ for $n\geq4$. For $n=3$, one computes directly
\[
D_3(r)=\sinh^2(r)-\cosh(r)\int^r_0\sinh(s)ds=\cosh(r)-1\geq 0.
\]
Let us summarize this discussion into a proposition.
\begin{proposition}\label{hyperbolicF}
Let $(\mathbb H^n, g_{hyp})=(\R^n, dr^2+\sinh(r)^2g_{\Sph^{n-1}})$, $n\ge3$, be the hyperbolic space with constant sectional curvature equal to $-1$. For any hyperbolic harmonic function $f$, the function
\[
N_{\mathbb H^n}(r)=\sinh(r)\cosh(r)\frac{E(r)}{F(r)}
\]
is nondecreasing in $r$ and $\lim_{r\to 0}N_{\mathbb H^n}(r)=\ord_p(f)$.
\end{proposition}

The proof of the identity
\[
\lim_{r\to 0}N_{\mathbb H^n}(r)=\ord_p(f)
\]
is entirely analogous to the proof of Proposition \ref{order}, so we omit the details.

\subsection{Hyperbolic three-sphere theorem}

We can now state the hyperbolic three-sphere theorem. The key to this result is to use the new \emph{hyperbolic} frequency function $N_{\mathbb H^n}(r)$ given in Proposition \ref{hyperbolicF} instead of the specialization of $\Ndef(r)$ in Theorem \ref{thm:Nmono} for $\phi(r)=\sinh(r)$.

\begin{theorem}\label{hyperbolic-Agmon}
Let $f$ be harmonic on $\mathbb H^n$, $n\ge3$. Then
\begin{equation}\notag
\frac{d^2}{dx^2}
\log\!\left(
\frac{F(r)}{\sinh^{n-1}r}
\right)\ge0,
\qquad
x=\log\tanh r.
\end{equation}
Equivalently, if $0<r_1<r<r_2$ and $\theta\in(0,1)$ is determined by
\begin{equation}\notag
\log\tanh r
=(1-\theta)\log\tanh r_1
+\theta\log\tanh r_2,
\end{equation}
then
\begin{equation}\label{eq:3sphere-hyperbolic}
\frac{F(r)}{\sinh^{n-1}r}
\le
\left(
\frac{F(r_1)}{\sinh^{n-1}r_1}
\right)^{1-\theta}
\left(
\frac{F(r_2)}{\sinh^{n-1}r_2}
\right)^\theta.
\end{equation}
\end{theorem}

\begin{proof}
By Proposition \ref{hyperbolicF}, we have
\[
\frac{d^2}{dx^2}\log\mathcal H_{\mathbb H^n}(r)
=2\frac{dN_{\mathbb H^n}}{dx}
=2\sinh r\cosh r\,N_{\mathbb H^n}'(r)\geq 0.
\]
The three-sphere inequality \eqref{eq:3sphere-hyperbolic} follows (after exponentiation) from the  convexity of $\log\mathcal H_{\mathbb H^n}$ with respect to the variable $x=\log\tanh r$.
\end{proof}


\begin{thebibliography}{ELMNPM}
\bibitem[Agm66]{Agm66} S.~Agmon, \emph{Unicit\'e et convexit\'e dans les probl\`emes diff\'erentiels}, S\'eminaire de Math\'ematiques Sup\'erieures, No.~13, Presses de l'Universit\'e de Montr\'eal, Montr\'eal, 1966.


\bibitem[Alm79]{Alm79} F. J. Almgren, Jr, \emph{Dirichlet’s problem for multiple valued functions and the regularity of
mass minimizing integral currents}. In \emph{Minimal submanifolds and geodesics} (Proc. Japan-
United States Sem., Tokyo, 1977), pages 1–6. North-Holland, Amsterdam-New York,
1979

\bibitem[Bru93]{Bru93} R.~Brummelhuis, \emph{Logarithmic convexity of $L^2$ norms for solutions of linear elliptic equations}, Indag. Math. (N.S.) \textbf{4} (1993), no.~4, 423--429.

\bibitem[CM11]{CM11} T.~H.~Colding and W.~P.~Minicozzi II, \emph{A Course in Minimal Surfaces}, Graduate Studies in Mathematics, vol.~121,
American Mathematical Society, Providence, RI, 2011.

\bibitem[DS22]{DS22} L. F. Di Cerbo, M. Stern, \emph{Price inequality and Betti number growth on manifolds without conjugate points}, Comm. Anal. Geom. \textbf{30} (2022), no. 2, 297-334.

\bibitem[DX84]{DX84} H.~Donnelly and F.~Xavier, \emph{On the differential form spectrum of negatively curved Riemannian manifolds},
Amer. J. Math. \textbf{106} (1984), no.~1, 169--185.

\bibitem[GL86]{GL86} N.~Garofalo and F.-H.~Lin, \emph{Monotonicity properties of variational integrals, $A_p$ weights and unique continuation},
Indiana Univ. Math. J. \textbf{35} (1986), no.~2, 245--268.

\bibitem[KM94]{KM94} J.~Korevaar and J.~L.~H.~Meyers, \emph{Logarithmic convexity for supremum norms of harmonic functions}, Bull. London Math. Soc. \textbf{26} (1994), no.~4, 353--362.

\bibitem[Rel40]{Rel40} F.~Rellich, \emph{Darstellung der Eigenwerte von $\Delta u+\lambda u=0$ durch ein Randintegral}, Math. Z. \textbf{46} (1940), 635--636.

\bibitem[Pet16]{Petersen} P. Petersen,  \emph{Riemannian Geometry}. Grad. Texts in Math., 171, Springer, Cham, 2016.

\bibitem[Pri83]{Pri83} P.~Price, \emph{A monotonicity formula for Yang--Mills fields}, Manuscripta Math. \textbf{43} (1983), no.~2--3, 131--166.

\end{thebibliography}
\end{document}